\documentclass{article}
\usepackage{amsmath}
\usepackage{amsfonts}

\newtheorem{theorem}{Theorem}

\newtheorem{definition}[theorem]{Definition}

\newtheorem{remark}[theorem]{Remark}

\newenvironment{proof}[1][Proof]{\noindent\textbf{#1.} }{\ \rule{0.5em}{0.5em}}
\begin{document}

\title{Local Fractional Calculus currently.\\
 A Mosaic of Definitions and the Quest for Unification.}
\author{Juan E. N\'{a}poles Valdes$^{1,2}$}
\maketitle

\begin{abstract}
In this paper, we present an overview of the development of one of the most dynamic areas of mathematics today: local differential operators of non-integer order. The underlying question is whether we are witnessing a period of consolidation and are on the way to a general formulation of this fascinating topic.
\end{abstract}

$\bigskip $

$^{1}$UNNE, FaCENA, Av. Libertad 5450, (3400) Corrientes, Argentina;

$^{2}$UTN, FRRE, French 414, (3500) Resistencia, Chaco, Argentina;

jnapoles@exa.unne.edu.ar; jnapoles@frre.utn.edu.ar



\section{Introduction.}

The current landscape of fractional calculus is defined by a vigorous debate over the best way to define and apply local operators. The discussion has evolved from a simple search for generalizations to a deep exploration of properties, kernels, and the connections between different definitions. 

On September 30, 1695 (see \cite{MR}), when Leibniz answers L'Hopital's question "What does $\frac{d^{n}}{dx^{n}}f(x) $ mean if $n=\frac{ 1}{2}$?" The Fractional Calculus is born, which in the last 50 years has had a remarkable development, with two large areas: global and local. This area refers to the generalization of differentiation and integration considering non-integer (i.e., fractional) orders. 

As we said, the local fractional calculus is a natural generalization of Calculus, and has an equally long mathematical history, we must say that it is not until recently, that it has started to play an important role in applications (\cite{Ba}). Actually, it was considered an area of Pure Mathematics and its foundations were not known or accessible to the public. Next year will be the 50th anniversary of the publication of the first text referring to its basic facts (\cite{OS}). From this fact, applications have multiplied in various fields such as epidemiology, various chemical and physical processes, population, electrical circuits, chaos, mechanics and many more (see \cite{KST}, \cite{ LLS}, \cite {MR} and \cite{P}).

A historical account of differential operators (local and global) is presented in \cite{Atan} (Chapter 1), beginning with Newton and ending with Caputo. In this book, the definition of a new, local derivative using one parameter is presented, and various applications are shown, illustrating its breadth and strength. In section 1.4 it is stated: “Therefore, we can conclude that both the Riemann-Liouville operator and the Riemann-Liouville Caputo operator are not derivatives, and therefore are not fractional derivatives, but fractional operators”, based on the fact that the classical notion of derivative rests on the limit of an incremental quotient and therefore is local, referred to the point. We want to point out on the other hand “the local fractional operator is not a fractional derivative" (see \cite{US}, p. 24), specifically: fractional derivatives are not derivatives, they are integrals! hence the term fractional operator before fractional derivative. The local operators, which will be our focus in this note, are new tools (sometimes not accepted due to their local character and affirming that they are just variants of the classical derivative) that have shown their strength in various applications (see \cite{ Atan2}). Many cases of local differential operators are known, we will see that many of these operators are particular cases of our 2020 definition and, more importantly, one of the exclusive features used to characterize classical fractional derivatives, the Leibniz Rule, does not It is exclusive to such operators, since we show a local operator that does not comply with said rule.

In this note we present the recent development of the so-called Non-Integer Order Local Calculus, which is the correct name (sometimes we use the name Generalized Calculus, although it does not illustrate the concept well). 

To facilitate reading, we will present several of the most commonly used definitions of derivatives and local integrals (interested readers can consult \cite{CT}, \cite{K} and \cite{VGLN}). It is clear that the definitions presented can be extended to the case where the order is in $(n,n+1)$, with $n \geq 1$.

We believe that readers are familiar with the classical Calculus, so we will not present it.

The following definitions represent an effort to create local operators that are either direct generalizations of the classical definition of the derivative, or that rely on concepts from fractal geometry to characterize the “smoothness" of a function.

\section{Specific Local Fractional Derivatives: A Direct and Unified Approach}

Although some attempts to use new local operators were recorded in the 1960s, it was not until Kolwankar and Gangal (\cite{KG96,KG97,KG98}), that what is now known as the local fractional calculus (also called local fractional calculus) appeared. fractal).

This definition stands out for its innovative and fundamental approach, as it is based on the fractal geometry of the function. Instead of using a kernel, the derivative is intrinsically linked to the fractal dimension of the function's graph.

The Kolwankar-Gangale derivative, also known as the “fractional derivative of fractal dimension", is defined as the limit of the quotient of an operator that depends on the fractal dimension of the function.

This approach is particularly useful for analyzing functions that are not differentiable in the classical sense but exhibit fractal behavior. It is a powerful tool in the study of irregular and complex phenomena.

This is the first formal definition of a local operator that generalizes the classical derivative, and we have the following definition:

\begin{definition} 
Let be a function $f$ such that $f:[0,1]\longrightarrow R$, if exists and is finite the following limit

\begin{equation} 
D^{ q }f(y)=\lim _{ x\to y } \frac { { d }^{ q }(f(x)-f(y)) }{ d(x-y)^{ q } } ,
\end{equation}

then we say that the local fractional derivative (LFD) of order $q$, at $x=y$, exists.
\end{definition}

Later, Parvate and Gangal (see \cite{PG}) introduced the definition of fractal derivative as follows, using it to study the fractal behavior of various functions:

\begin{equation} \label{e:0}
{ x }_{ 0 }{ D }_{ x }^{ \alpha  }f(y)=\frac { { d }^{ \alpha  }f(x) }{ { dx }^{ \alpha  } } \left( { x }_{ 0 } \right) =F-\lim _{ x\to { x }_{ 0 } } \frac { f(x)-f({ x }_{ 0 }) }{ { S }_{ F }^{ \alpha  }(x)-{ S }_{ F }^{ \alpha  }({ x }_{ 0 }) } ,
\end{equation}

where the right hand side is the notion of the limit by the points of the fractal set $F$. 

\begin{definition} 
The integral staircase function ${ S }_{ F }^{ \alpha  }(x)$ of order $\alpha$ for a set $F$ is given by:

\begin{equation} 
{ S }_{ F }^{ \alpha }\left( x \right) = \left\{ \begin{array}{lcc}
             { \gamma  }^{ \alpha  }\left[ F,a,x \right]  &   if  & x\ge a \\
             \\ { -\gamma  }^{ \alpha  }\left[ F,a,x \right]  &  if  & x<a 
             \end{array}
   \right.
\end{equation}

where $a$ is an arbitrary but fixed real number.
\end{definition}

and the mass function is defined in this way\\

\begin{definition} 
The mass function ${ \gamma  }^{ \alpha  }\left[ F,a,b \right]$ can written as (see \cite{Y11a} and \cite{Y11b}):

\begin{equation}
{ \gamma  }^{ \alpha  }\left[ F,a,b \right] =\lim _{ \delta \longrightarrow 0 }{ { \gamma  }_{ \delta  }^{ \alpha  }[F,a,b]=\frac { { \left( b-a \right)  }^{ \alpha  } }{ \Gamma \left( 1+\alpha  \right)  }  }.
\end{equation}
\end{definition}

Another version can be found at \cite{AC}:

\begin{equation}
{ x }_{ 0 }{ D }_{ x }^{ \alpha  }f(y)=\frac { { d }^{ \alpha  }f(x) }{ { dx }^{ \alpha  } } \left( { x }_{ 0 } \right) =\lim _{ x\rightarrow { x }_{ 0 }^{ \sigma  } }{ { D }_{ y,-\sigma  }^{ \alpha  } } \left[ \sigma \left( f(x)-f({ x }_{ 0 } \right) (x) \right] ,
\end{equation}

with $\sigma =\pm $ and ${ D }_{ y,-\sigma  }^{ \alpha  }$ is the Riemann-Liouville derivative.\\
The following notion can be found in \cite {Ch}:

\begin{equation}
{ x }_{ 0 }{ D }_{ x }^{ \alpha  }f(y)=\frac { { d }^{ \alpha  }f(x) }{ { dx }^{ \alpha  } } \left( { x }_{ 0 } \right) =\lim _{ x\rightarrow { x }_{ 0 } }{ \frac { f(x)-f({ x }_{ 0 }) }{ { x }^{ \alpha  }-{ x }_{ 0 }^{ \alpha  } }  }, 
\end{equation}

obtained from \eqref{e:0} under assumption ${ { x }^{ \alpha  }-{ x }_{ 0 }^{ \alpha  } }={ (x-{ x }_{ 0 }) }^{ \alpha  }$.

Yang's derivative (see \cite{Y11b,He2}) is a relatively generalized recent development of the conformable derivative. Its main innovation lies in the modification of the limit, which introduces a general function for the increment, giving it considerable flexibility to model different types of phenomena.

It allows us to address problems in which the behavior of the system's “memory" is not constant, but rather depends on a weight function or a power law more complex than that of the conformal operator. It is a powerful tool in the investigation of complex phenomena, such as those found in materials physics or biology.

The conformable derivative is one of the most elegant and straightforward proposals of the last decade and has gained widespread popularity due to its simplicity.

The conformable derivative of order $\alpha$ is a generalization of the classical derivative that uses a simple and straightforward limit (\cite{KHYS}).

This definition of local derivative opens a new direction of work, which remains very relevant in mathematical research due to its multiple applications.

So they presented the following definition(see also \cite{Ab}).

Thus, for a function $f:(0,\infty )\rightarrow R$ the conformable derivative of order $0<\alpha \le 1$ of $f$ at $t > 0$ was defined by

\begin{equation} \label{k:1}
{ T }_{ \alpha }f(t)=lim_{\varepsilon \rightarrow 0 }{ \frac { f(t+\varepsilon { t }^{ 1-\alpha })-f(t) }{ \epsilon } } ,
\end{equation}

and the fractional derivative at $0$ is defined as ${ T }_{ \alpha }f(0)=lim _{t\rightarrow 0 }{ { T }_{ \alpha }f(t)}$.

In a work from the same year (cf. \cite{K,Kat,kat}) another conformable derivative is defined in a very similar way. Let $f$ be a function of $(0,\infty )\rightarrow \mathbb{R}$, $t>0$ define the derivative of order $\alpha$ with $0<\alpha < 1$ as the expression  ${ { D }^{ \alpha } }f(t)=lim_{\varepsilon \rightarrow 0 }{ \frac { f(te^{ { \varepsilon t }^{ -\alpha } })- f(t) }{ \varepsilon } } $, of course, if ${ { D }^{ \alpha } }f(t)$ exists at some $(0,a)$ with $a>0$ then defines the derivative of order $\alpha$ at $0$ as ${ { D }^{ \alpha } }f(0)=lim_{t\rightarrow 0 }{ { { D }^{ \alpha } }f (t) }$. This is a very important definition because it unifies several other local derivatives into a single framework.

The Katugampola derivative is defined as a generalization of the conformal derivative, using an additional parameter. This definition is so flexible that by choosing specific values for the parameter, the classical derivative, the conformal derivative, and the Riemann-Liouville derivative can be obtained.

\cite{K} introduces a new twist when it defines a general derivative as follows, $f:\mathbb{R}\rightarrow \mathbb{R}$ is a function, $\alpha$ a real number, the derivative of fractional order can be thought of as ${ f } ^{ \alpha }(t)=lim_{ \varepsilon \rightarrow 0 }{ \frac { { f }^{ \alpha }(t+\varepsilon )-f(t) }{ { (t+\varepsilon ) } ^{ \alpha }-{ t }^{ \alpha } } } $.

Its main advantage is that it satisfies all the properties of the classical derivative (product rule, quotient rule, chain rule, etc.) in a simple and natural way, which facilitates its use in solving differential equations.

All these results, although they do not exactly coincide with the direction of our work, we present them so that readers have a more complete picture and because they have become relevant again in recent years.

\section{Post Kahlil derivative}

He gave a new fractal derivative in theis way (\cite{He}):

\begin{equation}
{ x }_{ 0 }{ D }_{ x }^{ \alpha  }f(y)=\frac { { d }^{ \alpha  }f(x) }{ { dx }^{ \alpha  } } \left( { x }_{ 0 } \right) =\lim _{ \Delta x\rightarrow { L }_{ 0 } }{ \frac { f(x)-f({ x }_{ 0 }) }{ K{ L }_{ 0 }^{ \alpha  } }  } . 
\end{equation}

Taking into account 

\begin{equation*}
{ H }^{ \alpha  }(F\cap (x,{ x }_{ 0 }))=(x-{ x }_{ 0 }^{ \alpha  })=\frac { K }{ \Gamma (1+\alpha ) } { L }_{ 0 }^{ \alpha  }.
\end{equation*}

This is the unified notation of \cite{GYK}.

In this address we have another definition (\cite{Yetall13} and \cite{Yetall15}), as follows:

\begin{equation}
{ D }_{  }^{ \alpha  }f(x)=\frac { { d }^{ \alpha  }f(x) }{ { dx }^{ \alpha  } } \left( { x }_{ 0 } \right) =\lim _{ x\rightarrow { x }_{ 0 } }{ \frac { { \Delta  }^{ \alpha  }\left[ f(x)-f({ x }_{ 0 }) \right]  }{ (x-{ x }_{ 0 })^{ \alpha  } }  } ,
\end{equation}

$ (x-{ x }_{ 0 })^{ \alpha  }$ is a measure fractal (\cite{Yetall15})  and  ${ \Delta  }^{ \alpha  }\left[ f(x)-f({ x }_{ 0 }) \right] \cong \Gamma (1+\alpha )\Delta \left[ f(x)-f({ x }_{ 0 }) \right]$.

In \cite{Yetall16} we have:

\begin{equation}
{ D }_{  }^{ \alpha  }f(x)=\frac { { d }^{ \alpha  }f(x) }{ { dx }^{ \alpha  } } \left( { x }_{ 0 } \right) =\lim _{ x\rightarrow { x }_{ 0 } }{ \frac { f(x)-f({ x }_{ 0 }) }{ (x-{ x }_{ 0 })^{ \alpha  } }  } .
\end{equation}

In 2018 we introduced a new local derivative, with a very distinctive property: when $\alpha \rightarrow 1$ we do not get the ordinary derivative. We call this derivative non-conformable, to distinguish it from the previous known ones, since when $\alpha \rightarrow 1$ the slope of the tangent line to the curve at the point is not preserved.

Be $\alpha \in (0,1]$ and define a continuous function $f:[t_{0},+\infty)\rightarrow \mathbb{R}$.

First, let's remember the definition of $_{1}N^{\alpha }f(t)$, a non conformable fractional derivative of a function in a point $t$ defined in [9] and that is the basis of our results, that are close resemblance of those found in classical qualitative theory.

\begin{definition}\label{d:1}
Given a function $f:[t_{0},+\infty )\rightarrow\mathbb{R}$, $t_{0}>0$. Then the $_{1}N$-derivative of $f$ of order $\alpha $\ is defined by $_{1}N^{\alpha }f(t)=\underset{\varepsilon \rightarrow 0}{\lim }\frac{f(t+\varepsilon e^{t^{-\alpha }})-f(t)}{\varepsilon }$ for all $t>0$, $\alpha \in (0,1)$. If $f$ is $\alpha$-differentiable in some $(0,a)$, and $\underset{t\rightarrow 0^{+}}{\lim }{_{1}N^{(\alpha )}f(t)}$ exists, then define $_{1}N^{\alpha }f(0)=\underset{t\rightarrow 0^{+}}{\lim }{_{1}^{(\alpha )}f(t)}$.
\end{definition}

If the above derivative of the function $x(t)$ of order $\alpha$ exists and is finite in $(t_{0},\infty )$, we will say that $x(t)$ is $_1 N^{\alpha}= N_{e^{t^{\alpha}}}$-differentiable in $I=(t_{0},\infty )$.

\begin{remark}
The use in the previous definitions, in particular in Definition \ref{d:1}, of the limit of a quotient of increments, similar to the classical case, unlike the classical fractional derivatives, which use a certain integral, allows us to interpret physically these local derivatives, and $_{1}N$ in the following way. Let us consider a certain point $P$ that moves in a straight line in $\mathbb{R}_{+}$. At time instants $t_{1}=t$ and $t_{2}=t+he^{t^{-\alpha }}$ where $h>0$ and $\alpha \in (0,1 ]$, we will denote by $S(t_{1})$ and $S(t_{2})$ the space traveled by the point $P$ in the time instants $t_{1}$ and $t_{2}$, respectively Then $\frac{S(t_{2})-S(t_{1})}{t_{2}-t_{1}}=\frac{S(t+he^{t^{ -\alpha }})- S(t)}{he^{t^{-\alpha }}}$ is the $_{1}N$-average velocity of point $P$ in the time interval $(t,he^{t^{-\alpha }})$. Passing to the limit in the previous expression, we have

\begin{equation*}
\underset{h\rightarrow 0}{Lim}\frac{S(t+he^{t^{-\alpha }})-S(t)}{he^{t^{-\alpha }}},
\end{equation*}

is the $_{1}N$-instantaneous velocity of point $P$ at any time $t>0$. Hence, we can conclude that a physical interpretation of the $_{1}N$-derivative is the velocity of point $P$ at instant $t_1$, and it makes sense, because the $_{1}N$-derivative is a local operator.
\end{remark}

\begin{remark}
A group of deficiencies are pointed out to the classical fractional derivatives, several of them significant for a “differential" operator. Such shortcomings are easily overcome by local operators, in particular by the $_{1}N$-derivative. So we have

\begin{theorem}
(See \cite{GLLMN}) Let $f$ and $g$ be $_{1}N$-differentiable at a point $t>0$ and $\alpha \in (0,1]$. Then

\begin{enumerate}
\item[a)] $_{1}N^{\alpha }(af+bg)(t)=a{\hspace{0.2cm} _{1}N^{\alpha }}(f)(t)+b{\hspace{0.2cm} _{1}N^{\alpha }}(g)(t).$
\item[b)] $_{1}N^{\alpha }(t^{p})=e^{t^{-\alpha }}pt^{p-1},\ p\in \mathbb{R}.$
\item[c)] $_{1}N^{\alpha }(\lambda )=0,\ \lambda \in 
\mathbb{R}.$
\item[d)] $_{1}N^{\alpha }(fg)(t)=f_{1}N^{\alpha }(g)(t)+g_{1}N^{\alpha }(f)(t).$
\item[e)] $_{1}N^{\alpha }(\frac{f}{g})(t)=\frac{gN_{1}^{\alpha
}(f)(t)-f_{1}N^{\alpha }(g)(t)}{g^{2}(t)}.$
\item[f)] If, in addition, $f$ is differentiable then $_{1}N^{\alpha
}(f)=e^{t^{-\alpha }}f^{\prime }(t).$
\item[g)] Being $f$ differentiable and $\alpha =n$ integer, we have $
_{1}N^{n}(f)(t)=e^{t^{-n}}f^{\prime }(t)$.
\end{enumerate}
\end{theorem}
\end{remark}

\begin{remark}
Readers can find additional details at \cite{KST} and \cite{P} and bibliography therein. 
\end{remark}

In the Second Lyapunov Method, for the study of the stability of perturbed motion, a basic tool is the well-known Calculus Chain Rule, now we will present the equivalent result, for $_{1}N^{\alpha }$, of said rule.

\begin{theorem}
(See \cite{GLLMN}) Let $\alpha \in (0,1]$, $g$ N-differentiable at $t>0$ and $f$ differentiable at $g(t)$ then $_{1}N^{\alpha }(f\circ g)(t)=f^{\prime}(g(t))_{1}N^{\alpha }g(t)$.
\end{theorem}

Next we define the integral operator, “corresponding" to the derivative $_{1}N$.

\begin{definition}
Given a function $f$ of real variable and an interval $[a,b]$ of the real line, we define the non conformable fractional integral of order $\alpha $ by the expression $_{1}J_{t_{0}}^{\alpha }f(t)=\int_{t_{0}}^{t}\frac{f(s)}{
e^{s^{-\alpha }}}ds$, with $t_0 \leq t$ in $[a,b]$.
\end{definition}

In the following theorem, the relationship between the operators $_{1}N$ and $_{1}J$ is presented, a result similar to that known from classical calculus.

\begin{theorem}
Let $f$ be $_1N$-differentiable function in $(t_{0},\infty )$ with $\alpha \in (0,1]$. Then for all $t>t_{0}$ we have

\begin{enumerate}
\item[a)] $_{1}J_{t_{0}}^{\alpha }\left( _{1}N^{\alpha
}f(t)\right) =f(t)-f(t_{0})$.
\item[b)] $_{1}N^{\alpha }\left( _{1}J_{t_{0}}^{\alpha }f(t)\right) =f(t)$.
\end{enumerate}
\end{theorem}
\begin{proof}
See \cite{NGL}
\end{proof}

This derivative, and some variants, proved useful in various application problems (see \cite{AFNRRS,FGNRS,FMNS,GLN,GLN2,GN,MMN,MN,NT,NRSi}).

\subsection{The N-derivative.}

In \cite{NGLK} a generalized derivative was defined as follows (see also \cite{FNRS,ZL}).

\begin{definition}	\label{d:2} 
Given a function $\psi :[0,+\infty )\rightarrow \mathbb{R}$. Then the N-derivative of $\psi $ of order $\alpha $\ is defined by	

	\begin{equation}\label{e:1}
		N_{F}^{\alpha }\psi (\tau )=\underset{\varepsilon \rightarrow 0}{\lim }\frac{\psi (\tau +\varepsilon F(\tau ,{\alpha }))-\psi (\tau )}{\varepsilon }
	\end{equation}
	
	for all $\tau >0$, $\alpha \in (0,1)$ being $F(\tau ,\alpha )$ is some
	function.
	
	If $\psi $ is N-differentiable in some $(0,\alpha )$, and $\underset{\tau \rightarrow 0^{+}}{\lim }N_{F}^{\alpha }\psi (\tau )$ exists, then
	define $N_{F}^{\alpha }\psi (0)=\underset{\tau \rightarrow 0^{+}}{\lim }N_{F}^{\alpha }\psi (\tau )$, note that if $\psi $ is differentiable, then $N_{F}^{\alpha }\psi (\tau )=F(\tau ,\alpha )\psi ^{\prime }(\tau )$ where $\psi ^{\prime }(\tau )$ is the ordinary derivative.
\end{definition}

\textbf{Examples}. Next we will see some particular non-conformable cases, obtained from certain $F$ kernels as a function of the well-known biparametric Mittag-Leffler function.

\begin{enumerate}
\item Mellin-Ross Function. So, we obtain 

\begin{equation*}
{ E }_{ t }(\alpha ,a)=t^{ \alpha  }E_{ 1,\alpha +1 }(at)=t^{ \alpha  }\sum _{ k=0 }^{ \infty  }{ \frac { (at)^{ k } }{ \Gamma (\alpha +k+1) }  } 
\end{equation*}

with $E_{1,\alpha +1}(.)$ the the aforementioned Mittag-Leffler two-parameter function. Hence, we have $\underset { \alpha \rightarrow 1 }{ \lim } N_{ E_{ t }(\alpha ,a) }^{ \alpha  }f(t)=f'(t)tE_{ 1,2 }(at)$, i.e.,

\begin{equation*}
N_{ E_{ t }(1,a) }^{ 1 }f(t)=f'(t)t\sum _{ k=0 }^{ \infty  }{ \frac { (at)^{ k } }{ \Gamma (k+2) }  }.
\end{equation*}

\item Robotov's Function. In this case we have

\begin{equation*}
R_{ \alpha  }(\beta ,t)=t^{ \alpha  }\sum _{ k=0 }^{ \infty  }{ \frac { \beta ^{ k }t^{ k(\alpha +1) } }{ \Gamma (1+\alpha )(k+1) }  } =t^{ \alpha  }E_{ \alpha +1,\alpha +1 }(\beta t^{ \alpha +1 })
\end{equation*}

like before, $E_{\alpha +1,\alpha +1}(.)$ is the the aforementioned Mittag-Leffler two-parameter function. So, we have $\underset { \alpha \rightarrow 1 }{ \lim } N_{ R_{ \alpha  }(\beta ,t) }^{ \alpha  }f(t)=f'(t)tE_{ 2,2 }(\beta t^{ 2 })$ and

\begin{equation*}
N_{ R_{ 1 }(\beta ,t) }^{ 1 }f(t)=\frac { f'(t)t }{ \Gamma (2) } \sum _{ k=0 }^{ \infty  }{ \frac { \beta ^{ k }t^{ 2k } }{ (k+1) }  }. 
\end{equation*}

\item Let $F(t,\alpha )={ E }_{ 1,1 }({ t }^{ -\alpha  })$. So, from Definition \ref{d:1}, we obtain  the derivative $_{ 1 }{ N }^{ \alpha  }f(t)$ defined in \cite{GLLMN} (and \cite{NGL}).

\item Be now $F(t,\alpha )={ { E }_{ 1,1 }({ t }^{ -\alpha  }) }_{ 1 }$, in this case we have $F(t,\alpha )={ \frac { 1 }{ { t }^{ \alpha  } }  }$, a new derivative with a remarkable propertie $\lim _{ t\rightarrow \infty  }{  N_{ \frac { 1 }{ { t }^{ \alpha  } } }^{ \alpha  }f(t) } =0$, i.e., goes to $0$ as $t$ goes to $+\infty$.

\item If we now take the development of function $E$ to order 1, we have $E_{ a,b }(t^{ -\alpha  })=1+\frac { 1 }{ { t }^{ \alpha  } } $. 
Then $\lim _{ t\rightarrow \infty  }{ { N }_{ 1+\frac { 1 }{ { t }^{ \alpha  } } }^{ \alpha  }f(t) } =f'(\infty)$, that is, it coincides with the classical derivative at infinity, if it exists.
\end{enumerate}

In this way, we can build a theoretical body for these local derivatives, similar to the existing one for the classical derivative and integral.

\begin{remark}
The generalized derivative defined above is not fractional (as we noted above), but it does have a very desirable feature in applications, its dual dependency on both $ \alpha $ and the kernel expression itself, with $ 0 <\alpha \leq 1 $ in \cite{KHYS} the conformal derivative is defined by putting $ F (t,\alpha) = t^{1- \alpha} $, while in \cite{GLLMN} the nonconforming derivative is obtained with $ F (t,\alpha) = e^{t^{-\alpha}} $ (see also \cite{NGL}). This generalized derivative, in addition to the aforementioned cases, contains as particular cases practically all known local operators and has proved its utility in various applications, see, for example, \cite{BKMN,CN,GKN,GLN,GNG,HNRF,KLN,LNV,MMN,MN,MSBFN,N21,N22,NGA,NJ,NQ,NQG,NRSi,NT,VKN,VNL}.
\end{remark}

\begin{remark}
One of the characteristics of this generalized derivative is the fact that $ N_{F}^{2\alpha}f(t) \neq N_{F}^{\alpha}\left( N_{F}^{\alpha } f(t)\right) $, that is, it is necessary to indicate successive derivatives in the second way. Obviously, if $ F \equiv 1 $, the ordinary derivative is obtained.
\end{remark}

\begin{remark} \label {r:n}
From the above definition, it is not difficult to extend the order of the derivative for $0 \le n - 1 <\alpha \le n $ by putting

\begin{equation} \label{e:n}
N_{F}^{\alpha }h(\tau) = \lim_{\varepsilon \rightarrow 0}\frac{h^{(n - 1)}(\tau + \varepsilon F(\tau, {\alpha})) - h^{(n - 1)}(\tau)}{\varepsilon}.
\end{equation}

If $h^{(n)}$ exists on some interval $I \subseteq \mathbb{R}$, then we have $N_{F}^{\alpha}h(\tau) = F(\tau, \alpha)h^{(n)}(\tau)$, with $0 \le n - 1 < \alpha \le n$.
\end{remark}

Slightly more recent, in \cite{FNRS} a notion of generalized fractional derivative is defined, which is general from two points of view:

1) Contains as particular cases, both conformable and non-conformable derivatives.

2) It is defined for any order $\alpha>0$.

Given $s \in \mathbb{R}$, we denote by $\left\lceil s \right\rceil$ the \emph{upper integer part} of $s$, i.e., the smallest integer greater than or equal to $s$.

\begin{definition}
Given an interval $I \subseteq (0,\infty)$, $f: I \rightarrow {\mathbb{R}}$, $\alpha \in \mathbb{R}^+$ and a continuous function positive $T(t,\alpha )$,
the \emph{derivative} $G_{T}^{\alpha }f$ of $f$ of order $\alpha$ at the point $t \in I$ is defined by

\begin{equation}
G_{T}^{\alpha }f(t)=\lim_{h\rightarrow 0} \frac{1}{h^{\left\lceil a \right\rceil}}\sum_{k=0}^{\left\lceil a \right\rceil}(-1)^{k}\binom{\left\lceil a \right\rceil}{k}f\big(t-khT(t,\alpha )\big).
\end{equation}
\end{definition}

In $2018$, a derivative operator is defined on the real line with a limit process as follows (see \cite{M}). For a given function $p$ of two variables, the symbol ${ D }_{ p }f(t)$ defined by the limit 
${ D }_{ p }f(t)=\lim_{\varepsilon \rightarrow 0 }{ \frac { f(p(t,\varepsilon ))-f(t) }{ \varepsilon } } $, as long as the limit exists and is finite, it will be called the derivative $p$ of $f$ at $t$ or the generalized derivative from $f$ to $t$ and, for brevity, we also say that $f$ is $p$-differentiable in $t$. In the case that it is a closed interval, we define the p-derivative at the extremes as the respective side derivatives. Starting from this definition, the derivative of order $\alpha$ of a function is constructed as the following limit:

\begin{equation}
{ { D }_{ p }^{ \alpha } }f(t)=\lim_{\varepsilon \rightarrow 0 }{ \frac { f(p(t,\varepsilon ,\alpha ))-f(t ) }{ \varepsilon } } ,\quad 0<\alpha <1,
\end{equation}

where it is understood that in the case $\alpha=1$ we have the ordinary derivative. It is clear that if $f$ is differentiable in $t$, then ${ { D }_{ p }^{ \alpha } }f(t)={ p }_{ h }(t,0,\alpha ) f'(t),\quad 0<\alpha <1$. Note that there are no sign restrictions on the function $p$ nor in its partial derivative ${ p }_{ h }(t,0,\alpha )$.

There is an additional detail that we want to point out, in \cite{NGLK} the following is pointed out.

We previously said that one of the features used to affirm that a derivative operator was fractional is that said operator violates the Leibniz Rule. The interesting thing is that a new local derivative can be built that does not comply with said rule, therefore, this non-compliance with the Leibniz Rule cannot be an essential condition for a derivative operator to be a fractional derivative, that is, there are still issues unresolved theories. To build such a derivative, we must be clear that failure to comply with this rule not only depends on the incremental quotient itself, but also on a factor that we must add to the incremented function, which will result in no symmetry in the final result.

Following the ideas of \cite{Z}, from \eqref{e:1} we can build the following derivative $\left( \alpha +\beta =1\right) $:

\begin{equation} \label{e:2}
DH_{\beta }^{\alpha }f(t):=\underset{\varepsilon \rightarrow 0}{\lim }\frac{H(\varepsilon ,\beta )f(t+\varepsilon F(t,\alpha ))-f(t)}{\varepsilon }
\end{equation}

with $H(\varepsilon ,\beta )\rightarrow k$ if $\varepsilon \rightarrow 0$.
In the case that $k\equiv 1$, we can consider two simple cases:
\begin{enumerate}
\item[I)] $H(\varepsilon ,\beta )=1+\varepsilon \beta $ as in \cite{Z} and so

\begin{equation*}
DL_{\beta }^{\alpha }f(t):=\underset{\varepsilon \rightarrow 0}{\lim }\frac{(1+\varepsilon \beta )f(t+\varepsilon F(t,\alpha ))-f(t)}{\varepsilon }.
\end{equation*}

If $F(t,\alpha )=e^{t^{-\alpha }}$, in this case we have a more general operator of $_1{N}$. So, we obtain:

\begin{equation} \label{e:3}
NL_{2}^{\alpha }f(t):=\underset{\varepsilon \rightarrow 0}{\lim }\frac{(1+\varepsilon \beta )f(t+\varepsilon e^{t^{-\alpha }})-f(t)}{\varepsilon }.
\end{equation}
\item[II)] $H(\varepsilon ,\beta )=1+\varepsilon \beta ^{r}$, $r>0$, in this way we
obtain

\begin{equation*}
DP_{\beta }^{\alpha }f(t):=\underset{\varepsilon \rightarrow 0}{\lim }\frac{(1+\varepsilon \beta ^{r})f(t+\varepsilon F(t,\alpha ))-f(t)}{\varepsilon }.
\end{equation*}

Refer to our $_{1}N$-derivative of \cite{GLLMN}  we have:

\begin{equation} \label{e:4}
NP_{2}^{\alpha }f(t):=\underset{\varepsilon \rightarrow 0}{\lim }\frac{(1+\varepsilon \beta ^{r})f(t+\varepsilon e^{t^{-\alpha }})-f(t)}{\varepsilon }.
\end{equation}

If $k\neq 1$, as $e^{x}=1+x+\frac{x^{2}}{2!}+...$we can take (as a first possibility):
\item[III)] $H(\varepsilon ,\beta )=E_{1,1}(\varepsilon \beta )$ and so we have

\begin{equation*}
DE_{\beta }^{\alpha }f(t):=\underset{\varepsilon \rightarrow 0}{\lim }\frac{E_{1,1}(\varepsilon \beta )f(t+\varepsilon F(t,\alpha ))-f(t)}{\varepsilon },
\end{equation*}

and regarding our $_1{N}$-derivative of \cite{GLLMN} we derive:

\begin{equation}\label{e:5}
NE_{\beta }^{\alpha }f(t):=\underset{\varepsilon \rightarrow 0}{\lim }\frac{E_{1,1}(\varepsilon \beta )f(t+\varepsilon e^{t^{-\alpha }})-f(t)}{\varepsilon }.
\end{equation}
\end{enumerate}

From \eqref{e:2} we can conclude with the following facts:

\begin{enumerate}
\item Being defined at a point, it is a local operator, similar to the classical derivative.
\item It is a derivative in the strict sense of the word, since its definition rests on the limit of an incremental quotient.
\item Does not satisfy the Leibniz Rule, so for \eqref{e:3} we have (the
results are similar for \eqref{e:4} and \eqref{e:5}):

\begin{equation*}
NL_{2}^{\alpha }\left[ f(t)g(t)\right] =\left( N_{2}^{\alpha
}f(t)\right) g(t)+f(t)\left( N_{F}^{\alpha }g(t)\right) ,
\end{equation*}

Also for \eqref{e:3} we have (again the calculations for \eqref{e:4} and \eqref{e:5} are very
similar):
\item If $\alpha =0$, $\beta =1$ then $N_{2}^{\alpha
}f(t)=N_{F}^{0}f(t)+f(t)=(1+e)f(t).$
\item If $\alpha =1$, $\beta =0$ then 
\begin{equation*}
N_{ 2 }^{ 1 }f(t)=N_{ { e }^{ { t }^{ -1 } } }^{ 1 }f(t)=\underset { \varepsilon \rightarrow 0 }{ \lim   } \frac { f(t+\varepsilon e^{ t^{ -1 } })-f(t) }{ \varepsilon  } =e^{ t^{ -1 } }\left[ \underset { \varepsilon \rightarrow 0 }{ \lim   } \frac { f(t+\varepsilon e^{ t^{ -1 } })-f(t) }{ \varepsilon e^{ t^{ -1 } } }  \right] =e^{ t^{ -1 } }f^{ \prime  }(t),
\end{equation*}
if $f$ is derivable.
\item Assuming that in \eqref{e:5} the limit exists, then we obtain

\begin{equation} \label{e:6}
NL_{\beta }^{\alpha }f(t)=N_{F}^{\alpha }f(t)+\beta f^{\prime }(t).
\end{equation}
\item Regrettably, ``we lose" the Chain Rule, valid for our $N$-derivative, so for $NL_{\beta }^{\alpha }$ we have:

\begin{equation*}
NL_{\beta }^{\alpha }\left[ f(g(t))\right] =N_{F}^{\alpha }f(g(t))+\beta
f(g(t)).
\end{equation*}

If $g(t)=t$, the above expression is a generalization of proportional derivative of \cite{AU}.
\item From \eqref{e:6} we derive that

\begin{equation*}
\underset{t\rightarrow \infty }{\lim}NL_{\beta }^{\alpha }f(t)=\underset{t\rightarrow \infty }{\lim}N_{F}^{\alpha }f(t)+\underset{t\rightarrow \infty }{\lim}\beta f^{\prime }(t)=f^{\prime }(t)+\beta f(\infty ).
\end{equation*}

From this we can conclude the following: if the term $\beta f(\infty )$ exists, then the derivative $N_{\beta }^{\alpha }f(t)$ is just a ``translation" of the classical derivative of the function, when $t$ tends to $\infty $, this is significant, because it does not affect the qualitative behavior of the ordinary derivative, which is important in the asymptotic study of the solutions of fractional differential equations with $NL_{\ beta}^{\alpha}$. In the event that said limit does not exist, it is impossible to study asymptotically the solutions of said $NL_{\beta }^{ \alpha}$-differential equations.

\item If we go back to the equation \eqref{e:2}, it is clear to readers that in the function $H(\varepsilon ,\beta )$ used, more general functions can be used, which would extraordinarily complicate the calculations and, from the point of view From a theoretical point of view, I would not add any essential detail. It is clear that this does not complete the analysis of which functions $H(\varepsilon ,\beta )$ can be considered in new local derivatives that, while maintaining locality, violate the Leibniz Rule. Therefore, as we said, this Rule cannot be considered an essential feature of a fractional derivative.
\end{enumerate}

\subsection{Most recent results.}

The Gohar fractional derivative (see \cite{GYD}) is a local derivative that generalizes the notion of the classical derivative through a limit incorporating a power factor. It is distinguished by being a tool that maintains the essential properties of traditional calculus, such as linearity and the product and quotient rules, but applied in a fractional-order context. Its approach is simple and straightforward, facilitating its use in various applications.

It was developed to offer a simple and powerful alternative for modeling complex phenomena in which integer-order derivatives are not sufficient. Its local structure makes it ideal for the analysis of systems that do not exhibit a significant long-term memory effect, but whose dynamics benefit from the generalization of the derivative's order.

The Vuk Stojiljković's D-derivative (cf. \cite{S}) is a local operator that flexibly generalizes the notion of the classical derivative. Its limit definition incorporates a differentiable function $g(t)$, which acts as a “weight" or “scale" factor. This function is key to adapting the operator to the specific characteristics of the problem being modeled. If $g(t) = t$, the conformal derivative is recovered, and if $g(t) = 1$, the classical derivative is obtained.

Its main advantage is that it provides a unified framework for conformal operators. By choosing the function g(t), a wide range of behaviors can be modeled, making it a versatile tool in mathematical physics and engineering.

The Beta Derivative (\cite{Atan,AD}), is a form of local fractional derivative that generalizes the conformable derivative. Its boundary definition includes an exponential weight function, allowing it to adjust the “memory effect" in a more controlled manner. It is presented as a powerful alternative for modeling systems with variable memory.

Its main advantage is that it allows greater flexibility in modeling complex phenomena, especially in anomalous diffusion and fluid mechanics, where memory effects can vary and decay over time. The parameter $\beta$ offers fine tuning that improves the accuracy in describing these systems.

The Conformable Ratio Derivative (in \cite{C}) is a generalization of the conformable derivative based on the idea of a “quotient" of derivatives, where the operator is defined by a ratio of differences of functions. This structure allows it to maintain the classical properties of the derivative while introducing a fractional order parameter that can be adjusted to model more complex phenomena.

It was developed to offer a more flexible alternative to existing conformal operators, maintaining the ease of calculation and the rules of traditional calculus, making it useful for solving differential equations.

In \cite{GG} the authors defined the derivative of H. Guebbai and M. Ghiat  is a local operator that generalizes the conformal derivative and the Katugampola derivative by introducing a general weight function into its definition. The choice of this weight function allows the operator to be adapted to the characteristics of the problem being modeled, making it a flexible and powerful tool.

Its objective is to unify several definitions of local derivatives into a single framework, allowing greater flexibility and versatility in modeling complex phenomena.

The derivative of Almeida et al. (2016)

This derivative is a local operator that generalizes the notion of a classical derivative through a limit incorporating a general function. Its main characteristic is that it is defined as a generalization of Khalil's conformable derivative, including a weight function that modifies the increment in the limit.

The objective of this operator is to provide a more flexible tool for modeling physical and biological phenomena, since the weight function can be chosen to describe the specific behavior of each system.\\

The derivative of A. Kajouni, A. Chafiki, K. Hilal, and M. Oukessou (2021)

This is a local fractional derivative based on a modification of the limit of the classical derivative, incorporating a weighting factor that depends on a general function. This definition allows it to generalize several already known local fractional operators, making it a unifying tool.

Its objective is to provide a broader framework for the study of local derivatives, allowing researchers to choose the most appropriate weight function for the problem they are solving, thereby increasing the versatility and accuracy of modeling.\\

The Tempered Derivative (cf. \cite{VCa,VCb})

This is one of the most influential generalizations, especially in physics and finance. The tempered derivative addresses a limitation inherent to classical fractional operators such as the Riemann-Liouville or Caputo operators.

While classical operators are based on a ``memory" that slowly decays with a power law, the tempered derivative introduces an exponential factor into its definition. This factor ``tempers" or moderates the system's memory, causing long-range correlations to decay more rapidly. It is the ideal tool for modeling phenomena that have a short- to medium-term memory, but not an ``infinite memory."

This feature makes it invaluable in modeling stochastic processes and diffusion equations that exhibit ``semi-heavy tails". For example, it is used to describe diffusion in porous media, turbulence in the atmosphere, or risk models in finance, where past effects do not have a permanent influence.

Other Generalizations and Advanced Derivatives

The constant evolution of the field has led to more sophisticated definitions that seek to address very specific problems or unify theories.\\

The Composite $N$ Derivative (see \cite{GNV})

Unlike the sequential application of operators, the ``composite $N$ derivative" is defined as a local operator that satisfies the fundamental properties of the classical derivative, such as the product rule. Its name comes from its structure, which is a ``composition" of a classical derivative with a power factor.

Its main advantage is that it provides a generalization of the standard derivative operator that is easier to handle analytically, as it preserves the basic rules of calculus. This makes it an attractive tool for solving fractional-order differential equations.\\

The Fractional Omega Derivative (cf. \cite{CNC})

Concept: This derivative, also known as the ``$\Omega$ derivative", is a generalization that relies on a function $\Omega(t)$ that defines the ``rhythm" of the differentiation. The definition of this derivative uses the function $\Omega(t)$ to modify the kernel of the classical fractional operator, offering great flexibility. It is a powerful tool for modeling systems whose memory dynamics change over time.

Its main advantage is that it allows modeling systems with varying memory properties, which is not possible with classical fractional-order derivatives that have a fixed kernel.\\

The Multi-Index Fractional Derivative (see \cite{VLNS})

While standard $N$-derivatives apply to functions a kernel with an unique expression, the multi-index derivative extends to kernels of multiple variables (sub-kernels). It uses a vector of fractional orders $\alpha=(\alpha_{1},\alpha_{2},,...,\alpha_{n})$ for a function of $n$ variables, where each component correspondents to a derivative of fractional order with respect to kernel $i$ of variable $i$. 

It is fundamental for modeling diffusion phenomena and memory processes in higher-dimensional spaces, such as the diffusion of contaminants in a $3D$ medium or the behavior of complex physical systems.\\

The Fractional $M$-derivative

Proposed by Sousa and Oliveira, the $M$-derivative is a local operator intrinsically linked to the Mittag-Leffler function, a function of great importance in fractional calculus that is often considered the ``exponential" of fractional calculus.

The fractional $M$-derivative is defined as a generalization of the classical derivative, but the scaling factor in its limit definition involves the Mittag-Leffler function. This gives it the ability to model behaviors that arise naturally in fractional systems.

The $M$-derivative was designed to provide a tool that, while local, also connects directly to the language of fractional calculus through the Mittag-Leffler function, facilitating the analysis of its properties and applications. It is a powerful tool for studying the dynamics of memory systems.

\section{Conclusions}

In this paper, we have presented a sketch of the latest developments obtained in the Non-Integer Order Calculus. Of course, they are not all, for example in \cite{VLNS} a multi-index derivative is presented that generalizes the previous definitions and includes as a particular case the derivative presented in \cite{RFCDM}.

Various applications of several of the aforementioned operators can be consulted in \cite{LNV2,NRMF,NR,NT}.

All of the above shows that this topic is a fruitful field and has not finished giving us good results. 

Further discussions about local differential operators can be found in \cite{BGHHN,PRBN}.

\end{document}